\documentclass[a4paper,11pt]{amsart}
\usepackage{graphicx}
\usepackage[ansinew]{inputenc}    % accents i altres
\usepackage[all]{xy}
\usepackage{color}

\usepackage[export]{adjustbox}%Para ajustar imágenes a la izqueirda o derecha

\usepackage{amssymb,amscd}

\usepackage{mathrsfs,mathtools}
\usepackage[shortlabels]{enumitem}%Para las opciones de enumerate como [(i)], shortlabels es para escribir i en lugar de roman etc.

\newtheorem{thm}{Theorem}[section]

\theoremstyle{definition}

\theoremstyle{definition}
\newtheorem{rem}[thm]{Remark}

\theoremstyle{definition}

\def\C{\mathbb C}

\def\dim{\operatorname{dim}}

\def\codim{\operatorname{codim}}
\def\sign{\operatorname{sign}}

\def\id{\operatorname{id}}

\def\sign{\operatorname{sign}}

\newcommand{\RR}{\mathbb{R}}
\newcommand{\CC}{\mathbb{C}}
\newcommand{\QQ}{\mathbb{Q}}

\newcommand{\eqA}{\mathscr{A}}

\newcommand{\OO}{\mathcal{O}}

\newcommand{\CCzero}[1]{(\mathbb{C}^{#1},0)}

\newcommand{\CCS}[1]{(\mathbb{C}^{#1},S)}
\newcommand{\GS}[2]{(\mathbb{C}^{#1},S)\rightarrow(\mathbb{C}^{#2},0)}

\newcommand{\Alt}{\textnormal{Alt}}
\theoremstyle{plain}
\newtheorem{theorem}{Theorem}[section]
\newtheorem{lemma}[theorem]{Lemma}
\newtheorem{corollary}[theorem]{Corollary}
\newtheorem{proposition}[theorem]{Proposition}

\theoremstyle{definition}
\newtheorem{definition}[theorem]{Definition}
\newtheorem{conjecture}[theorem]{Conjecture}
\newtheorem{example}[theorem]{Example}

\theoremstyle{remark}

\newtheorem*{note}{Note}

\begin{document}

\author{R. Giménez Conejero and
J.J.~Nu\~no-Ballesteros}

\title[The image Milnor number and excellent unfoldings]
{The image Milnor number and excellent unfoldings}

\address{Departament de Matemàtiques,
Universitat de Val\`encia, Campus de Burjassot, 46100 Burjassot
SPAIN}
\email{Roberto.Gimenez@uv.es}
\email{Juan.Nuno@uv.es}

\thanks{The first named author has been partially supported by MCIU Grant FPU16/03844. The second named author has been partially supported by MICINN Grant PGC2018--094889--B--I00 and by GVA Grant AICO/2019/024}


\begin{abstract} 
We show three basic properties on the image Milnor number $\mu_I(f)$ of a germ $f\colon\CCS{n}\rightarrow\CCzero{n+1}$ with isolated instability. First, we show the conservation of the image Milnor number, from which one can deduce the upper semi-continuity and the topological invariance for families. Second, we prove the weak Mond's conjecture, which says that $\mu_I(f)=0$ if and only if $f$ is stable. Finally, we show a conjecture by Houston that any family $f_t\colon\CCS{n}\rightarrow\CCzero{n+1}$ with $\mu_I(f_t)$ constant is excellent in Gaffney's sense. By technical reasons, in the two last properties we consider only the corank 1 case.
\end{abstract}

\maketitle

\section{Introduction}

The image Milnor number is an invariant introduced by D. Mond in \cite{Mond1991} for map-germs $f\colon\CCS{n}\rightarrow\CCzero{n+1}$ with isolated instability. Based on results of Lê and Siersma (cf. \cite{Trang1987} and \cite{Siersma1991}), he showed that the image of a stable perturbation of $f$ has the homotopy type of a wedge of $n$-spheres and that the number of such spheres is independent of the stabilisation. He called this number, denoted by $\mu_I(f)$, the image Milnor number by its analogy with the classical Milnor number $\mu(X,0)$ of a hypersuface $(X,0)$ with isolated singularity. In order to ensure the existence of a stabilisation of $f$, it was considered in \cite{Mond1991} only the case where $(n,n+1)$ are nice dimensions in the sense of Mather (cf. \cite{Mather1971}). But when $f$ has corank 1, it always admits a stabilisation, so there is no reason to put any restriction on the dimensions in such case.

In this paper we will show three basic results on the image Milnor number. The first one is in Section \ref{sect:conservation} and is about the conservation of the image Milnor number. If $F(x,u)=(f_u(x),u)$ is any $r$-parameter unfolding of $f$, then for all $u$ in $\CC^r$ close to 0,
\[ 
\mu_I(f)=\beta_n(X_u)+\sum_{y\in X_u}\mu_I(f_u;y),
\]
where $\beta_n(X_u)$ is the number of spheres (i.e., the $n$-th Betti number) of the image $X_u$ of $f_u$ and $\mu_I(f_u;y)$ is the image Milnor number of $f_u$ at $y\in X_u$ (see Theorem \ref{conservation}). Two immediate consequences of this conservation are that $\mu_I(f)$ is upper semi-continuous (Corollary \ref{upper}) and also that $\mu_I(f)$ is a topological invariant for families of germs (Corollary \ref{top-triv-2}).

The second result is what we call the weak Mond's conjecture in Section \ref{sect:wmc}. The original Mond's Conjecture, cf. \cite{Mond1991}, says that
\[
\eqA_e\text{-}\codim(f)\le \mu_I(f),
\]
with equality if $f$ is weighted homogeneous. Here $\eqA_e\text{-}\codim(f)$ is the $\eqA_e$-codimension of $f$, that is, the minimal number of parameters in a versal unfolding of $f$. This conjecture is known to be true for $n=1,2$ (see \cite{deJong1991,Mond1991} for $n=2$ and \cite{Mond1995} for $n=1$) but it remains open for $n\ge 3$. In our weak version of the conjecture (Theorem \ref{wmc}) we prove that $\mu_I(f)=0$ if and only if $\eqA_e\text{-}\codim(f)=0$ or, equivalently, $f$ is stable (by Mather's criterion of infinitesimal stability). Since we use here the results of Houston on the image computing spectral sequence (cf. \cite{Houston2010}), we have to restrict ourselves to the corank 1 case.

Finally, in Section \ref{sect:hc} we prove a conjecture by Houston in \cite{Houston2010} relative to excellent unfoldings. Following Gaffney in \cite{Gaffney1993}, an origin-preserving one-parameter unfolding $F(x,t)=(f_t(x),t)$ is excellent if it admits a stratification by stable types such that the parameter axes are the only 1-dimensional strata (see \ref{def:excellent} for a more precise definition). Excellent unfoldings play an important role in the theory of equisingularity of mappings. In fact, if the unfolding is excellent, then the polar multiplicity theorem of Gaffney states that the Whitney equisingularity of family is equivalent to the constancy of the polar multiplicities associated to all the strata in the source and target (see \cite{Gaffney1993}). The conjecture of Houston is that the constancy of the image Milnor number $\mu_I(f_t)$ is also a sufficient condition for an unfolding to be excellent. We prove this result in Theorem \ref{hc} (also provided that $f$ has corank 1).

We refer to the book \cite{Mond-Nuno2020} for basic definitions and properties about singularities of mappings, such as stability, finite determinacy, versal unfoldings, etc.

\medskip
\noindent
\emph{Acknowledgements:}
The authors thank G. Peñafort-Sanchis for his help with the results of Section \ref{sect:conservation} and the anonymous referee for the careful reading and valuable suggestions.

\section{Conservation of the Image Milnor number}\label{sect:conservation}

%To begin with we need a result that is believed to be true but there is no prove in the literature, the conservation of the image Milnor number or the upper semi-continuity of itself (as a consequence of the conservation). The idea is simple and depicted in \ref{fig:semicontinuity}. 
%
%
%
%\begin{figure}[htbp]
%	\centering\
%		\includegraphics[width=1.00\textwidth]{Excellent_images/semicontinuity.png}
%		\caption{A real representation of the conservation of the image Milnor number with consecutive deformations of an original map germ, from left to right.}
%		\label{fig:semicontinuity}
%\end{figure}
%
%
%
%
%
%
%In this paper we give a prove of it using the following Siersma's result preceded by a bit of notation.
We recall the definition of the Milnor fibration (see \cite[Theorem 4.8 and Theorem 5.8]{Milnor1968}). 
Let $g:\CCzero{n+1}\rightarrow\CCzero{}$ be a holomorphic non-zero function which defines a hypersurface $X=g^{-1}(0)$ in $(\CC^{n+1},0)$ with arbitrary singularities (either isolated or non-isolated). We fix a Whitney stratification on $X$. We denote by $B_\epsilon$ the closed ball of radius $\epsilon$ centred at $0$ in $\CC^{n+1}$, with boundary $S_\epsilon=\partial B_\epsilon$ and interior $\mathring B_\epsilon=B_\epsilon\setminus S_\epsilon$.

A \emph{Milnor radius} is a number $\epsilon>0$ such that $S_{\epsilon'}$ is transverse to $X$, for all $\epsilon'$ such that $0<\epsilon'\le\epsilon$. This implies that $X\cap B_\epsilon$ is homeomorphic to the cone on $X\cap S_\epsilon$.

Once we have fixed $\epsilon>0$, there exists $\eta>0$ such that 
$$g:g^{-1}(\mathring{D}_\eta)\cap B_\epsilon\rightarrow \mathring{D}_\eta$$
is a locally trivial fiber bundle over $\mathring{D}_\eta\setminus\left\{0\right\}$. 
Here, $\mathring{D}_\eta$ is the open disk of radius $\eta$ centred at 0 in $\CC$. 
The choice of $\eta$ has to be made in such a way that for all $t$ such that $0<|t|<\eta$, then $t$ is a regular value of $g$ and also $S_\epsilon$ is transverse to $g^{-1}(t)$. This is called the \emph{Milnor fibration} and the fibres are called \emph{Milnor fibres}.

Now we consider an $r$-parameter deformation of $g$, that is, a holomorphic germ $G\colon(\CC^{n+1}\times\CC^{r},0)\rightarrow(\CC,0)$ written as $G(y,u)=g_u(y)$ and such that $g_0=g$. Then $G$ defines a hypersurface $\mathcal X=G^{-1}(0)$ in $(\CC^{n+1}\times\CC^{r},0)$ which is a deformation of $X$. We assume that $\mathcal X$ also has a Whitney stratification whose restriction to $\{u=0\}$ coincides with that of $X$.

\begin{definition}[cf. {\cite[page 2]{Siersma1991}}]
We say that the deformation $G$ is \emph{topologically trivial over the Milnor sphere $S_\epsilon$} if, for $\eta$ and $\rho$ small enough,
\begin{equation}\label{top-triv-sphere}
\begin{CD}(S_\epsilon \times \mathring{B}_\rho) \cap G^{-1}(\mathring{D}_\eta)@>(G,\id)>>& \mathring{D}_\eta\times \mathring{B}_\rho \end{CD}
\end{equation}
is a stratified submersion with strata $\left\{0\right\}\times \mathring{B}_\rho$ and $(\mathring{D}_\eta \setminus \left\{0\right\})\times \mathring{B}_\rho$ on $\mathring{D}_\eta\times \mathring{B}_\rho $ and the induced stratification on $(S_\epsilon \times \mathring{B}_\rho) \cap G^{-1}(\mathring{D}_\eta)$.
\end{definition}

Since we have a Whitney stratification on $\mathcal X$, the restriction of \eqref{top-triv-sphere} to each stratum in the target is a locally trivial $C^0$- fibration,  by the Thom-Mather first isotopy lemma (cf. \cite[Theorem 5.2]{Gibson1976a}).

\begin{theorem}[cf. {\cite[Theorem 2.3]{Siersma1991}}]\label{siersma}
With the notation above, let $G$ be a deformation of $g$ which is topologically trivial over a Milnor sphere. Let $u\in \mathring{B}_\rho$ and suppose that all the fibres of $g_u$ are smooth or have isolated singularities except for one special fibre $X_{u}\coloneqq g_u^{-1}(0)\cap B_\epsilon$. Then $X_u$ is homotopy equivalent to a wedge of spheres of dimension $n$ and its number is the sum of the Milnor numbers over all the fibres different from $X_u$.
\end{theorem}

\begin{example} The condition that $G$ is topologically trivial over a Milnor sphere is necessary in Theorem \ref{siersma}. For instance, consider $G:(\C^3\times\C,0)\to(\C,0)$ given by $G(x,y,z,u)=xy-u$. For $u\ne0$, $X_{u}=g_u^{-1}(0)\cap B_\epsilon$ has not the homotopy type of a wedge of 2-spheres (in fact, it has the homotopy type of $S^1$).
\end{example}

Let $f\colon\CCS{n}\rightarrow\CCzero{n+1}$ be an $\eqA$-finite germ, that is, with finite $\eqA_e$-codimension. By the Mather-Gaffney criterion (see e.g. \cite[Theorem 4.5]{Mond-Nuno2020}), this is equivalent to that $f$ has isolated instability. In particular, $f$ is finite and hence, its image is an analytic hypersurface $(X,0)$ in $(\CC^{n+1},0)$. We take a holomorphic function $g:\CCzero{n+1}\rightarrow\CCzero{}$ such that $g=0$ is a reduced equation for $X$. We will assume that either $(n,n+1)$ are nice dimensions in Mather's sense (cf. \cite[page 208]{Mather1971}) or $f$ has corank 1. In both cases, $X$ has a natural stratification given by the stable types. This stratification is analytically trivial and hence a Whitney stratification (see \cite[Corollary 7.5]{Mond-Nuno2020}).

Consider now an unfolding $F\colon(\CC^n\times\CC^r,S\times\{0\})\to(\CC^{n+1}\times\CC^r,0)$ of $f$. Write $F(x,u)=(f_u(x),u)$, as usual, with $f_0=f$. We denote by $(\mathcal X,0)$ the image of $F$ in $(\CC^{n+1}\times\CC^r,0)$ and choose a holomorphic function $G\colon(\CC^{n+1}\times\CC^r,0)\to(\CC,0)$ such that $G=0$ is a reduced equation of $\mathcal X$ and $g_0=g$, where $g_u(y)=G(y,u)$. We also consider in $\mathcal X$ the natural Whitney stratification by stable types, which has the property that its restriction to $u=0$ coincides with the stratification of $X$.
We say that $G$ is a deformation of $g$ \emph{induced} by the unfolding $F$.

\begin{lemma}  Let $f$ be $\eqA$-finite such that either $(n,n+1)$ are nice dimensions or $f$ has corank 1. Any deformation $G$ induced by an unfolding $F$ is topologically trivial over a Milnor sphere. 
\end{lemma}

\begin{proof} The proof of this lemma is basically the same that appears in \cite[proof of Theorem 1.4]{Mond1991} in the particular case that $F$ is a stabilisation of $f$. On one hand, $f$ is $\eqA$-finite, hence it has isolated instability, so $f$ is locally stable on $S_\epsilon$. On the other hand, $g$ is regular on $S_\epsilon$ by definition of Milnor radius. Since $S_\epsilon$ is compact, we can assume, after shrinking $\rho$ if necessary, that $f_u$ is locally stable on $S_\epsilon$ and $g_u$ has no critical points on $S_\epsilon$, for all $u\in \mathring{B}_\rho$. Now we prove that 
\[
\begin{CD}(S_\epsilon \times \mathring{B}_\rho) \cap G^{-1}(\mathring{D}_\eta)@>(G,\id)>>& \mathring{D}_\eta\times \mathring{B}_\rho \end{CD}
\]
is a stratified submersion.

In fact, let $(y,u)\in (S_\epsilon \times \mathring{B}_\rho) \cap G^{-1}(\mathring{D}_\eta)$. If $y\in X_u$ then $f_u$ is stable at $y$ and hence, $F$ is (analytically) trivial in a neighbourhood of $(y,u)$. This implies that the induced stratification in $(S_\epsilon \times \mathring{B}_\rho)\cap \mathcal X$ is also (analytically) trivial in a neighbourhood of $(y,u)$. In particular, the map
\[
\begin{CD}(S_\epsilon \times \mathring{B}_\rho) \cap \mathcal X @>0\times\id>>& \{0\}\times \mathring{B}_\rho \end{CD}
\]
is a stratified submersion at $(y,u)$. Otherwise, if $y\notin X_u$, then $y$ is a regular point of $g_u$ and hence, $(y,u)$ is a regular point of $(G,\id)$. It follows that
\[
\begin{CD}(S_\epsilon \times \mathring{B}_\rho) \cap G^{-1}(\mathring{D}_\eta\setminus\{0\}) @>(G,\id)>>& (\mathring{D}_\eta\setminus\{0\}) \times \mathring{B}_\rho \end{CD}
\]
is a submersion at $(y,u)$.
\end{proof}

It follows from Theorem \ref{siersma} that for all $u$ small enough, $X_u$ is homotopy equivalent to a wedge of spheres of dimension $n$ and its number is the Betti number
\[
\beta_n(X_u)=\sum_{y\in B_\epsilon\setminus X_u}\mu(g_u;y).
\]

Since $f$ is finite, it has finite singularity type and hence, there exists a stable unfolding $F$ of $f$.
The bifurcation set $B(F)$ is the set germ of parameters $u\in \mathring{B}_\rho$, for $\rho>0$ small enough, such that $f_u$ is not a locally stable mapping. When $(n,n+1)$ are nice dimensions or $f$ has corank 1 we know that $B(F)$ is a proper analytic subset of $(\CC^r,0)$ (see \cite[Propositions 5.5 and 5.6]{Mond-Nuno2020}).

\begin{definition} Let $f$ be $\eqA$-finite such that either $(n,n+1)$ are nice dimensions or $f$ has corank $1$. Take $F$ a stable unfolding of $f$ and $u\in \mathring{B}_\rho\setminus B(F)$, for $\rho$ small enough. The mapping $f_u$ is called a \emph{stable perturbation} of $f$, its image $X_u$ is called the \emph{disentanglement} of $f$ and the number of spheres $\beta_n(X_u)$ is called the \emph{image Milnor number} and is denoted by $\mu_I(f)$.
\end{definition}

The image Milnor number $\mu_I(f)$ is well defined, that is, it is independent of the choice of the parameter $u$, of the representatives and of the stable unfolding $F$ (see \cite[Section 8]{Mond-Nuno2020} for details). 

\begin{rem}
When $(n,n+1)$ are not nice dimensions and $f$ has corank $>1$, the definition of $\mu_I(f)$ can be done analogously by taking Mather's canonical stratification of the image instead of the stratification by stable types and taking a parameter $u$ such that $f_u$ is topologically stable instead of stable. However, we will not consider these cases in this paper.
\end{rem}

\begin{rem}
A  \textit{stabilisation} of $f$ is a $1$-parameter unfolding $F(x,t)=(f_t(x),t)$ of $f$ with the property that $f_t$ is a locally stable mapping for all $t\ne0$ close to the origin. A stabilisation of $f$ always exists when $(n,n+1)$ are nice dimensions or $f$ has corank $1$ (cf. \cite[Corollary 5.4]{Mond-Nuno2020}). Moreover, given a second stable unfolding $F'(x,u)=(f_u(x),u)$ of $f$ we can take the sum of unfoldings
\[
F''(x,u,t)=(f_u(x)+f_t(x)-f(x),u,t),
\]
which is also stable. For all $t\ne0$, $f_t$ is stable, so $(0,t)\notin B(F'')$ and hence $\mu_I(f)=\beta_n(X_t)$, where $X_t$ is the image of $f_t$. This is in fact the definition of $\mu_I(f)$ given originally by D. Mond in \cite[Theorem 1.4]{Mond1991} in terms of a stabilisation instead of a stable unfolding.
\end{rem}

The following property can be seen as the conservation of the image Milnor number.
%
%Following the idea of this result we can prove that the image Milnor number is conservative. The \ref{fig:siersmadef2} gives the main idea and objects of the proof.

\begin{theorem}\label{conservation}
Let $f$ be $\eqA$-finite such that either $(n,n+1)$ are nice dimensions or $f$ has corank $1$.  
Let $F$ be any unfolding of $f$. Take $u\in \mathring{B}_\rho$, with $\rho>0$ small enough. 
Then,
$$ \mu_I(f)=\beta_n(X_u)+\sum_{y\in X_u}\mu_I(f_u;y),$$
where $\mu_I(f_u;y)$ is the image Milnor number of $f_u$ at $y\in X_u$.
\label{guillesiersma}
\end{theorem}

\begin{proof}
By taking the sum of $F$ with a stable unfolding, we can assume that $F$ is itself stable.
Since $f$ is $\eqA$-finite, $f$ has isolated instability at the origin by the Mather-Gaffney criterion. 
This implies that $f_u$ has only finitely many unstable singularities which we denote by 
$y_1,\dots,y_k\in X_u$ and hence,
\[
\sum_{y\in X_u}\mu_I(f_u;y)=\sum_{i=1}^k\mu_I(f_u;y_i).
\]
Also by Theorem \ref{siersma}, $g_u$ has only finitely many (isolated) critical points on $B_\epsilon\setminus X_u$, which we denote by $z_1,\dots,z_m$, so that
\[
\beta_n(X_u)=\sum_{j=1}^m\mu(g_u;z_j).
\]

For each $i=1,\dots,k$ we choose a Milnor ball $B_{\epsilon_i}$ for $g_u$ at $y_i$ contained in $B_\epsilon$. Analogously, for each $j=1,\dots,m$ we choose also a Milnor ball $B_{\delta_j}$ for $g_u$ at $z_j$ contained in $B_\epsilon\setminus X_{u}$. We will assume that the balls $B_{\epsilon_1},\dots,B_{\epsilon_k},B_{\delta_1},\dots,B_{\delta_m}$ are pairwise disjoint (see fig. \ref{fig:siersmadef2}, right).

\begin{figure}[htbp]
		\includegraphics[width=1.50\textwidth, center]{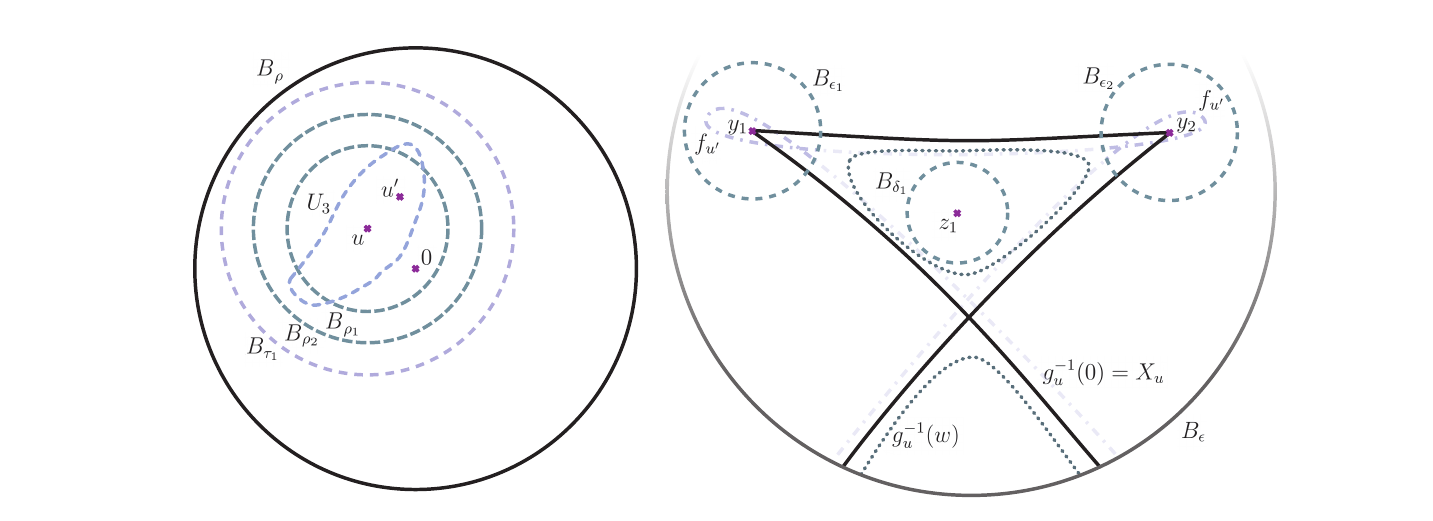}
	\caption{Balls in the parameter space (left) and in the target (right).}
	\label{fig:siersmadef2}
\end{figure}

Again by Theorem \ref{siersma}, for each $i=1,\dots,k$, there exists an open ball  $\mathring B_{\rho_i}$ centered at $u$ and contained in $\mathring{B}_\rho$ such that
\[
\mu_I(f_{u};y_i)=\beta_n(X_{u'}\cap B_{\epsilon_i})=\sum_{w\in B_{\epsilon_i}\setminus X_{u'}}\mu(g_{u'};w),
\]
for all $u'\in \mathring B_{\rho_i}\setminus B(F)$. We set $U_1=\mathring{B}_{\rho_1}\cap\dots \cap \mathring{B}_{\rho_k}$ (see fig. \ref{fig:siersmadef2}, left).

For each $j=1,\dots,m$, $z_j$ is an isolated critical point of $g_u$ and $X_u\cap B_{\delta_j}=\emptyset$. By the conservation of the Milnor number of a function, there exists another open ball $\mathring B_{\tau_j}$ centered at $u$ and contained in $\mathring{B}_\rho$ such that
\[
\mu(g_u;z_j)=\sum_{w\in B_{\delta_j}}\mu(g_{u'};w),
\]
and also $X_{u'}\cap B_{\delta_j}=\emptyset$, for all $u'\in \mathring B_{\tau_j}$.  As above, we set $U_2=\mathring B_{\tau_1}\cap\dots \cap \mathring B_{\tau_m}$.

Consider the compact set 
\[
K=B_\epsilon\setminus \left(\bigcup_{i=1}^k \mathring B_{\epsilon_i}\cup\bigcup_{j=1}^m \mathring B_{\delta_j}\right).
\]
Since $g_u$ has no critical points on $K\setminus X_u$, there exists another open neighbourhood $U_3$ of $u$ in $\mathring{B}_\rho$ such that $g_{u'}$ has no critical points on $K\setminus X_{u'}$,
for all $u'\in U_3\setminus B(F)$.

Finally, again by Theorem \ref{siersma}, 
\begin{align*}
\mu_I(f)&=\beta_n(X_{u'})=\sum_{w\in B_\epsilon\setminus X_{u'}}\mu(g_{u'};w)\\
&=\sum_{i=1}^k \sum_{w\in B_{\epsilon_i}\setminus X_{u'}}\mu(g_{u'};w)+
\sum_{j=1}^m \sum_{w\in B_{\delta_j}}\mu(g_{u'};w)\\
&=\sum_{i=1}^k\mu_I(f_u;y_i)+\sum_{j=1}^m\mu(g_u;z_j),
\end{align*}
for all $u'\in U_1\cap U_2\cap U_3\setminus B(F)$.

\end{proof}

A straightforward consequence of Theorem \ref{conservation} is that the image Milnor number is upper semi-continuous.

\begin{corollary}\label{upper}
With the conditions and notation of Theorem \ref{conservation}, 
\[
\mu_I(f)\ge \mu_I(f_u;y),
\] 
for all $y\in X_u$.
\end{corollary}

The upper semi-continuity of $\mu_I(f)$ has been also obtained by Houston in \cite[Theorem 4.3]{Houston2010} but in the particular case that $f$ has corank 1 and either $s(f_u)\le d(f_u)$ or $s(f_u)$ and $ d(f_u)$ are constant (see Section 3 for the definitions of $s(f_u)$ and $d(f_u)$).
\medskip

Another consequence of the conservation is the topological invariance of the image Milnor number for unfoldings. We say that an unfolding $F$ is \emph{topologically trivial} if it is $\eqA$-equivalent as an unfolding to the constant unfolding. That is, if there exist homeomorphisms $\Phi$ and $\Psi$ which are unfoldings of the identity in $(\CC^n,S)$ and $(\CC^{n+1},0)$, respectively, such that $\Psi\circ F\circ \Phi^{-1}=f\times\id$.

\begin{corollary}\label{top-triv-1}
With the conditions and notation of Theorem \ref{conservation}, if $F$ is topologically trivial
$$ \mu_I(f)=\sum_{y\in X_u}\mu_I(f_u;y).$$
\end{corollary}

\begin{proof} Write $F(x,u)=(f_u(x),u)$, $\Phi(x,u)=(\phi_u(x),u)$ and $\Psi(y,u)=(\phi_u(y),u)$. Then $\psi_u\circ f_u\circ\phi_u^{-1}$, for all $u$. Hence, $X_u$ is homeomorphic to $X$ which is contractible.
\end{proof}

A particular case is when $F$ is good in the sense of Gaffney \cite[Definition 2.1]{Gaffney1993}. Roughly speaking it means that $F$ has isolated instability uniformly. We will assume that $F$ is a one-parameter unfolding which is origin-preserving, that is, $f_t(S)=\{0\}$, for all $t$.

\begin{definition}\label{good}
We say that an origin-preserving one-parameter unfolding $F(x,t)=(f_t(x),t)$ is \emph{good} if there exists a representative $F:U\to W\times T$, where $U$ is an open neighbourhood of $S\times\{0\}$ in $\CC^n\times\CC$ and $W,T$ are open neighbourhoods of the origin in $\CC^{n+1},\CC$ respectively, such that
		\begin{enumerate}[(i)]
			\item $F$ is finite,
			\item $f_t^{-1}(0)=S$, for all $t\in T$,
			\item $f_t$ is locally stable on $W\setminus\left\{0\right\}$, for all $t\in T$.
		\end{enumerate}
\end{definition}

\begin{corollary}\label{top-triv-2}
If $F$ is a topologically trivial and good unfolding of an $\eqA$-finite germ $f$, then
$\mu_I(f_t)$ is constant for the family of germs $f_t\colon\GS{n}{n+1}$.
\end{corollary}

\section{Weak Mond's conjecture}\label{sect:wmc}

In this section we prove the weak version of Mond's conjecture. We first recall the definition of the multiple point spaces $D^k(f)$ following \cite[Proposition-definition 2.5]{Nuno2017}.

%We will need more notation that we introduce now.
%
%Consider $h:\GS{n}{n+1}$, we will denote by $\tilde{h}$ the disentanglement map of $h$, i.e. the map $f_t|_{f_t^{-1}(B_\epsilon)}$ for $\epsilon$ a Milnor radius. Also $D^k(h)$, the multiple point space of $h$, can be defined in the following way

\begin{definition}  The \emph{$k$th-multiple point space} $D^k(f)$ of a mapping or a map germ $f$ is defined as follows:
\begin{itemize}
\item
 Let $f\colon X\rightarrow Y$ be a locally stable mapping between complex manifolds. Then $D^k(f)$ is equal to the Zariski closure of the set of points  $\left(x^{(1)},\dots,x^{(k)}\right)$ in $X^k$ such that $f\left(x^{(i)}\right)=f\left(x^{(j)}\right)$, but $x^{(i)}\neq x^{(j)}$ for all $i\neq j$.

\item When $f\colon (\CC^n,S)\rightarrow(\CC^p,0)$ is a stable germ, then $D^k(f)$ is defined analogously but in this case it is a set germ in $\left((\CC^n)^k,S^k\right)$.

\item
Let $f\colon(\CC^n,S)\rightarrow(\CC^p,0)$ be finite and let $F(x,u)=(f_u(x),u)$ be a stable unfolding of $f$. Then $D^k(f)$ is the complex space germ in $\left((\CC^n)^k,S^k\right)$ given by
$$
D^k(f)=D^k(F)\cap\left\{u=0\right\}.
$$
\end{itemize}
\end{definition}

The fact that $D^k(f)$ is independent of the choice of the stable unfolding $F$ can be found in \cite[Lemma 2.3 and Proposition-definition 2.5]{Nuno2017}. In the particular case of a corank 1 mono-germ $f\colon(\CC^n,0)\to(\CC^{p},0)$, we have explicit equations for the multiple point spaces $D^k(f)$. These are given by the so called divided differences of $f$, which were introduced by Mond in \cite[Section 3]{Mond1987}.
The multi-germ version (also in corank 1) can be found in \cite[page 555]{Marar1989}.

Suppose $F\colon X\times U\to Y\times U$ is a mapping of the form $F(x,u)=(f_u(x),u)$. Then $D^k(F)$ contains only $k$-tuples $\left(x^{(1)},u,\dots,x^{(k)},u\right)$ with the same parameter $u$. So, it is more convenient to embed $D^k(F)$ in $X^k\times U$ by identifying such a $k$-tuple with the point $\left(x^{(1)},\dots,x^{(k)},u\right)$.

Given a mapping $f\colon X\to Y$, the symmetric group $\Sigma_k$ acts on $D^k(f)$ by permuting the points $x^{(i)}$. This induces also an action of $\Sigma_k$ on the homology and the cohomology of $D^k(f)$. In general, if $G$ is a subgroup of $\Sigma_k$ acting linearly on a vector space $V$, then the $G$-\emph{alternating part} of $V$ is the subspace
\[
\left\{v\in V:\ \sigma v=\sign(\sigma)v, \text{ for all $\sigma\in G$ }\right\},
\]
and if the group $G$ is $\Sigma_k$ we omit the group and denote this by $V^\Alt$.

\bigskip
 The following definition is due to Houston \cite[Definition 3.9]{Houston2010}:

\begin{definition}\label{mu k}
Let $f\colon(\CC^n,S)\to(\CC^{p},0)$, $n<p$, be $\eqA$-finite of corank 1 and let $F(x,t)=(f_t(x),t)$ be a stabilisation of $f$. We set the following notation:
\begin{itemize}
\item $s(f)=|S|$, the number of branches of the multi-germ; 
\item $d(f)=\sup\{k:\  D^k(f_t)\neq\emptyset\}$, where $f_t$ is a stable perturbation of $f$.
\end{itemize}
The \textit{$k$-th alternating Milnor number of $f$}, denoted by $\mu_k^\Alt(f)$, is defined as
\[
\mu_k^\Alt(f)= \begin{cases}\dim_\QQ H^\Alt_{n+1-k+1}\left(D^k(F),D^k(f_t);\QQ\right),\quad \text{if $k\leq d(f)$,} \\ \\
	\left|\displaystyle\sum_{l=d(f)+1}^{s(f)}(-1)^l{s(f)\choose l}\right|,\quad  \text{if $k=d(f)+1$ and $s(f)>d(f)$,}\\ \\
	0,\quad \text{otherwise.}\end{cases}
\]
\end{definition}

The value of $\mu_k^\Alt(f)$ when $k=d(f)+1$ and $s(f)>d(f)$ can be simplified in the following way:
\[\left|\sum_{l=d+1}^{s}(-1)^l{	s\choose	l	}\right|={	s-1\choose	d}.
	\]
%This equality can be proven easily by using elementary properties of binomial numbers. Another observation is that the inequality $s(f)>d(f)$ only can happen when $d(f)$ has the maximal possible value. For instance, for a germ $f\colon(\CC^n,S)\to(\CC^{n+1},0)$, we have $s(f)>d(f)$ only when $d(f)=n+1$.
This equality can be proven easily by using elementary properties of binomial numbers. Another useful property is the following lemma, which gives a relation between $s(f)$ and $d(f)$.
\begin{lemma}
In terms of the Definition \ref{mu k}, the inequality $s(f)>d(f)$ can only happen when $d(f)$ has the maximal possible value.
\end{lemma}  

\begin{proof}
Suppose that the maximal possible value for $d(f)$, for map-germs of the type $f\colon(\CC^n,S)\to(\CC^{p},0)$, is $m$. We will assume, for the sake of the contradiction, that for a map-germ $f$ in this pair of dimensions $d(f)<m$ but $d(f)<s(f)$.

Let $k$ be $\textnormal{min}\left\{s(f),m\right\}$, hence $d(f)<k$. If we prove that $D^k(f_t)$ is not empty, for $f_t$ a stable perturbation of $f$, we arrive to a contradiction. To simplify the argument, assume that we have a stabilization $F(x,t)=(f_t(x),t)$ of $f$, so that $f_t$ is a stable perturbation for every $t\neq 0$, and this unfolding is inside another unfolding $\mathcal{F}$ of $f$ that is stable as a map-germ, i.e., $\mathcal{F}(x,t,u)=(f_{t,u}(x),t,u)$ such that $f_{t,0}=f_t$. 

Given that $k\leq s(f)$, necessarily $D^k(f)$ has at least a point. In fact, since $f$ has $s(f)$ branches passing through the origin of $\CC^p$, any subset of $S$ with $k$ distinct points determines a point in $D^k(f)$. However, notice that 
\begin{align*}D^k(f)=&D^k(\mathcal{F})\cap \left\{t=0,u=0\right\}\textnormal{ and}\\
D^k(f_{t_0})=&D^k(\mathcal{F})\cap\left\{t=t_0,u=0\right\}= D^k(F)\cap\left\{t=t_0\right\}
\end{align*}
as well, by definition. Hence, if $D^k(F)$ has bigger dimension than $D^k(f)$ we finish because then the intersection with $\left\{t=t_0\right\}$ will contain at least a point, otherwise the dimensions would be equal.

In fact, since $k \le m$ and $f$ is $\eqA$-finite, by \cite[Theorem 2.14]{Marar1989} it follows that $\textnormal{dim }D^k(f) = nk - p(k - 1)$. The stabilisation $F$ of $f$ is also $\eqA$-finite (see \cite[Exercise 5.4.2]{Mond-Nuno2020}), so that $\textnormal{dim }D^k(F) =  (n + 1)k - (p + 1)(k - 1)$. Since both sets are non-empty, $\textnormal{dim }D^k(F) > \textnormal{dim }D^k(f)$, and this finishes the proof.

%Since $k\le m$, the dimension of $D^k(f)$ is $nk-p(k-1)$ and the dimension of $D^k(F)$ is $(n+1)k-(p+1)(k-1)$, which is bigger, because both are non-empty (cf. \cite[Theorem 2.14]{Marar1989} for the dimensions and \cite[Exercise 5.4.2]{Mond-Nuno2020} to see that the stabilisation is $\eqA$-finite). This finishes the proof.
\end{proof}

For instance, for a germ $f\colon(\CC^n,S)\to(\CC^{n+1},0)$, we have $s(f)>d(f)$ only when $d(f)=n+1$.

The motivation for the definition of $\mu_k^\Alt(f)$ is the following result by Houston which shows that, for a corank 1 germ $f\colon(\CC^n,S)\to(\CC^{n+1},0)$, the image Milnor number $\mu_I(f)$ is equal to the sum of all the alternating Milnor numbers. 

\begin{proposition}[cf. {\cite[Definition 3.11]{Houston2010}}]\label{mu}
Let $f\colon(\CC^n,S)\to(\CC^{n+1},0)$ be $\eqA$-finite of corank 1. Then,
	$$\mu_I(f)=\sum_k\mu_k^\Alt(f).$$
	\end{proposition}

The proof is based on the analysis of the image computing spectral sequence associated to the multiple point spaces. Moreover, the above equality can be taken as a definition when we consider the more general situation of a germ $f\colon(\CC^n,S)\to(\CC^{p},0)$, with $p\ge n+1$. In that case, $\mu_I(f)$ can be also interpreted in terms of the homology of the distenganglement of $f$ (see \cite[Remark 3.12]{Houston2010} for details).

On the other hand, the following result, due to Wall, will be crucial. Suppose $g\colon\CCzero{n+1}\rightarrow\CCzero{ }$ has isolated singularity at $0$. Let $U=\OO_{n+1}/J_g$ be the Milnor algebra, where $J_g$ is the Jacobian ideal, generated by the partial derivatives $\partial g/\partial y_i$, $1\le i\le n+1$. Denote by $X_t=g^{-1}(t)\cap B_\epsilon$ the Milnor fiber, where $0<\delta\ll\epsilon\ll 1$ and $0<|t|<\delta$. We assume $G$ is a finite group of automorphisms of $\CCzero{n+1}$ that leaves $g$ invariant. This implies that we have induced actions of $G$ on $X_t$ and on $U$.

\begin{theorem}[see {\cite[Theorem of page 170]{Wall1980a}}]\label{wall iso}
With the above notation, we have an isomorphism of  $\CC G$-modules
$$ H^n(X_t;\CC)\cong U \otimes_\CC \Lambda^{n+1}(\CC^{n+1})^*,$$
where $\Lambda^{n+1}(\CC^{n+1})^*$ is the $(n+1)$th exterior power of the dual $(\CC^{n+1})^*$.
\end{theorem}

Obviously, the same is true if we replace $\CC$ by $\QQ$ and consider homology instead of cohomology.
%
%
%This definition can be simplified, but we are keeping this expression because of its original motivation as the result of a convergent spectral sequence  (cf. \cite{Houston2010} for details). The simplification comes from the fact that
%
%	\[\left|\sum_{l=d+1}^{s}(-1)^l{	s\choose	l	}\right|={	s-1\choose	d},
%	\]
%	
%	which can be proven using the recursive formula of the binomial numbers. 
%	
%	We can perform another simplification as we are working with germs $f:\CCS{n}\rightarrow \CCzero{n+1}$. This one relies on the fact that $\codim (X \cap Y)\leq \codim(X)+\codim(Y)$, so if we apply this to an inductive reasoning with the branches of a multi-germ we get that $s(f)>d(f)$ only if $d(f)=n+1$.
%	
%	The previous definition provides us tools to work with in order to make computations related to the image Milnor number. This comes from the following proposition.
%	
%	\begin{proposition}[\cite{Houston2010}]\label{mu}
%	The image Milnor number of $f$, denoted $\mu_I(f)$, is
%	
%	$$\mu_I(f)=\sum_k\mu_k^\Alt(f).$$
%	\end{proposition}
%	
%	\begin{remark}
%Of course, in the case of corank one, this could be taken as a definition instead of the one by Mond in \cite{Mond1991}. This is what Houston does in \cite{Houston2010}.
%	\end{remark}

	We are now able to state and prove the following essential lemma about the structure of the alternating homology of the multiple point spaces:
	
	\begin{lemma}\label{milnor icis}
	Let $f:\CCS{n}\rightarrow \CCzero{n+1}$ be unstable of corank 1, $\eqA$-finite  and which admits a $1$-parameter stable unfolding $F(x,t)=(f_t(x),t)$. Take $f_t$ a stable perturbation of $f$ and $k=2,\dots,d(f)$. Then, $H_{n-k+1}^\Alt(D^k(f_t);\QQ)\ne0$  if and only if $D^k(f)$ is singular. 
Furthermore, if $H_{n-k+1}^\Alt(D^k(f_t);\QQ)\ne0$, then $H_{n-k'+1}^\Alt(D^{k'}(f_t);\QQ)\ne0$ for all $k'=k,\dots,d(f)$.	\end{lemma}

%\textcolor{red}{Yo creo que es falso, por ejemplo, para el punto cuádruple de $\CC^3$ en $\CC^4$ se tiene que $D^2(f)$ y $D^3(f)$ son lisos.}

	\begin{proof}
	To prove the first part we begin by studying the case $S=\left\{0\right\}$. We use the Marar-Mond criterion, \cite[Theorem 2.14]{Marar1989}. Since $F$ is stable, $D^k(F)$ is smooth and $D^k(f)$ is a hypersurface in $D^k(F)$ with isolated singularity and with Milnor fibre $D^k(f_t)$. Moreover, the symmetric group $\Sigma_k$ leaves invariant the defining equation of $D^k(f)$ in $D^k(F)$. By Theorem \ref{wall iso}, we have an isomorphism of  $\CC \Sigma_k$-modules
$$H^{n-k+1}(D^k(f_t);\CC)\cong U \otimes_\CC \Lambda^{n-k+2}V^*,$$
where $U$ is the Milnor algebra of $D^k(f)$ in $D^k(F)$ and $V=T_0 D^k(F)$ is the tangent space of $D^k(F)$ at the origin. If $D^k(f)$ is singular, then $U\ne0$ and contains the constants. Now we 
will see that these constants, after tensoring with $\Lambda^{n-k+2}V^*$, are contained in the alternating part.

%To do so we need to know how $\Sigma_k$ acts on the right part of the tensor product,  $\Lambda^{n-k+2}(\CC^{n-k+2})^*$. Note first that $(\CC^{n-k+2})^*$ is the dual of the tangent space at a point of $D^k(F)$, so if we know how $\Sigma_k$ acts on that tangent space we are done.

From \cite[Proposition 2.3]{Marar1989} we can take $\Sigma_k$-invariant equations for $D^k(F)$ 
in $\CC^{n}\times\CC^k$. Since $F$ has corank 1, we assume that $D^k(F)$ is embedded in $\CC^{n}\times\CC^k$ with coordinates $x_1,\dots,x_n,y_1,\dots,y_k$ and that $\Sigma_k$ acts by permuting $y_1,\dots,y_k$. It follows that the tangent space $V$ has $\Sigma_k$-invariant linear equations of the form
	\[ a_i(y_1+\dots+y_k)+\sum_{j=1}^n b_{i,j} x_j=0, \textnormal{ for }i=1,\dots,n.
\]
Hence, we can split $V$ as $V=V_1\oplus V_2$, where 
\begin{align*}
V_1&=\left\{x_1=0,\dots,x_n=0,y_1+\dots+y_k=0\right\},\\
V_2&=V\cap\left\{y_i=y_j, 1\leq i<j\leq k\right\}.
\end{align*}
If $\omega_1,\dots,\omega_\ell$ is any basis of $V_2^*$, then
	\[ \lambda=(dy_1-dy_2)\wedge\dots\wedge(dy_{k-1}-dy_k)\wedge \omega_1\wedge\dots\wedge \omega_\ell
\]
generates $\Lambda^{n-k+2}V^*$ and is $\Sigma_k$-alternating. This shows $H^{n-k+1}(D^k(f_t);\CC)$ has non-zero alternating part in the mono-germ case.

Suppose now that $S$ is any finite set. Let $D^k_1(F),\dots,D^k_m(F)$ be the connected components of $D^k(F)$. Each $D^k_i(F)$ is a mono-germ at a point  $\left(z^{(1)},\dots,z^{(k)},0\right)\in S^k\times\{0\}$. We also denote by  $D^k_1(f),\dots,D^k_m(f)$ the connected components of $D^k(f)$ such that $D_i^k(f)\subset D_i^k(F)$ for any $i$. 

As $D^k(f)$ is singular, without loss of generality we can suppose that $D_1^k(f)$ is singular. Assume that $D_1^k(f)$ is a mono-germ at $\left(z^{(1)},\dots,z^{(k)}\right)\in S^k$ and let $G\leq\Sigma_k$ be the stabilizer of this point. By following the same argument as in the mono-germ case, but with $D_1^k(F)$, $D_1^k(f)$ and $G$ instead of $D^k(F)$, $D^k(f)$ and $\Sigma_k$, respectively, we find a non-zero element $v$ in the homology of $D^k_1(f_t)$ which is $G$-alternating.

Now for each $i=1,\dots,m$ we choose a permutation $\sigma_i\in\Sigma_k$ that takes $D_1^k(f_t)$ into $D_i^k(f_t)$. We claim that $\omega=\sum_i \sign(\sigma_i)\sigma_i v$ is a non-zero element in the homology of $D_k(f_t)$ which is alternating. 

Let $\tau\in\Sigma_k$. For each $i=1,\dots,m$, $\tau$ takes $\sigma_i\left(z^{(1)},\dots,z^{(k)}\right)$ into some other $\sigma_{j(i)}\left(z^{(1)},\dots,z^{(k)}\right)$, where $j(i)=1,\dots,m$. We can write $\tau\sigma_i$ as $\tau\sigma_i=\sigma_{j(i)}\left(\sigma_{j(i)}^{-1}\tau\sigma_i\right)$, and $\left(\sigma_{j(i)}^{-1}\tau\sigma_i\right)\in G$. Hence,
\begin{align*}
			\tau\omega&=\tau\sum_i \sign(\sigma_i)\sigma_iv\\
			&=\sum_i \sign(\sigma_i)^2\sign(\tau)\sign(\sigma_{j(i)})\sigma_{j(i)}v\\
			&=\sign(\tau)\sum_i\sign(\sigma_{j(i)})\sigma_{j(i)}v.
	\end{align*}
But if $j(i_1)=j(i_2)$, for some $i_1\ne i_2$, then 
\[g=\left(\tau\sigma_{i_1}\right)^{-1}\left(\tau\sigma_{i_2}\right)=\sigma_{i_1}^{-1}\sigma_{i_2}
\]
 is in $G$ as it fixes $\left(z^{(1)},\dots,z^{(k)}\right)$. We have $\sigma_{i_2}=\sigma_{i_1}g$ and both $\sigma_{i_1}$ and $\sigma_{i_2}$ take $D_1^k(f_t)$ to the same component, which is absurd. Hence, $\tau\omega=\sign(\tau)\omega$.
	 
	 This concludes the proof that if $D^k(f)$ is singular, then $H^{n-k+1}(D^k(f_t);\CC)$ has non-zero alternating part. The converse is obvious: if $D^k(f)$ is smooth then $H^{n-k+1}(D^k(f_t);\CC)=0$, which has no alternating part. 

	For the second part, take $k$ such that $D^k(f)$ is singular. Then $D^k(f)$ is a subspace of $\left((\CC^{n})^k,S^k\right)$, with coordinates $x_i^{(j)}$, with $i=1,\dots,n$ and $j=1,\dots,k$, 
and whose equations are the divided differences, which we represent by $\phi_1,\dots,\phi_{r}$ with $r=(n+1)(k-1)$. Moreover, $D^k(f)$ has codimension $r$ and, 
by the Jacobian criterion, the Jacobian matrix $A$ of the functions $\phi_1,\dots,\phi_{r}$ has rank less than $r$ at some point in $S^k$. 

Now $D^{k+1}(f)$ is defined in $\left((\CC^{n})^{k+1},S^{k+1}\right)$, by adding $n$ new coordinates 
$x_1^{(k+1)},\dots,x_n^{(k+1)}$ and $n+1$ new equations $\phi_{r+1},\dots,\phi_{r+n+1}$. Since the old equations do not depend on the new variables, the Jacobian matrix of $\phi_1,\dots,\phi_{r+n+1}$ is
	\[
	\left(\begin{array}{c|c}
	A & 0 \\\hline
	\ast & B\\
	\end{array}\right),
\]
where $B$ is the Jacobian matrix of the new equations with respect to the new variables. Obviously, this matrix has rank $<r+n+1$ at some point in $S^{k+1}$ and thus, $D^{k+1}(f)$ is also singular, since it has codimension $r+n+1$. We can proceed recursively for $D^{k'}(f)$, with $k\leq k'\leq d(f)$.
%	
%	if $f$ is non-stable then we can use twice \cite[Theorem 2.14]{Marar1989} to see $D^k(f)$ as an hypersurface of $D^k(F)$, which is smooth, and then deduce that $D^k(f)$  is an isolated complete intersection singularity (ICIS) or void for $2 \leq k \leq n+1$. Now we are in the conditions of \ref{wall iso}, so we can consider the defining equation of $D^k(f)$ and deduce that the Milnor fiber, which can be taken $D^k(\tilde{f})$, has alternating part if, and only if, $D^k(f)$ is a non-smooth ICIS. To prove the converse note that the constants will appear in the left hand side of the tensorial product of the theorem.
%	
%	Now note that if, for such a $k$,
%	
%	$$ D^k(f)=V(\phi_1,\dots,\phi_r)$$
%	
%	where $r=nk-(n-k+1)$, then the matrix $d(\phi_1,\dots,\phi_r)$ has not maximum rank at some point. Hence $ D^{k+1}(f)$, if $k+1\leq d(f)$, is such that
%	
%	$$ D^{k+1}(f)=V(\phi_1,\dots,\phi_r,f(\underline{x})-f(\underline{x}'))$$
%
%where $n+1$ new equations and $n$ new variables, $\underline{x}'$, have been added. Then the matrix
%
%	\[\left(\begin{array}{c|c}
%	d(\phi_1,\dots,\phi_r) & 0 \\\hline
%	\ast & -df_{\underline{x}'}
%	\end{array}\right)
%\]
%
%has again rank not equal to the maximum in some points so $D^{k+1}(f)$ is singular and the same reasoning using \ref{wall iso} can be applied here. A recurrent argument finishes the proof.
%\qed
	\end{proof}

Let $f\colon\CCS{n}\rightarrow \CCzero{n+1}$ be $\eqA$-finite. Take $F$ be a stable unfolding and choose $G\colon(\CC^{n+1}\times\CC^r,0)\to(\CC,0)$ such that $G(y,u)=0$ is a reduced equation of the image of $F$. The relative Jacobian ideal is the ideal $J_y(G)$ generated by the partial derivatives of $G$ with respect to the variables $y_1,\dots,y_{n+1}$.

\begin{lemma} We have:
\[
\mu_I(f)=0\Longleftrightarrow G\in \sqrt{J_y(G)}.
\]
\end{lemma}

\begin{proof} We follow the notation of Section 2. If $G\in\sqrt{J_y(G)}$ then $V\left(J_y(G)\right)\subseteq V(G)$. Hence, for any $(y,u)$ such that $y$ is a singular point of $g_u$, we have $g_u(y)=0$. In particular, for $u\notin B(F)$,
\[
\mu_I(f)=\beta_n(X_u)=\sum_{y\in B_\epsilon\setminus X_u}\mu(g_u;y)=0.
\]

Conversely, if $G\notin\sqrt{J_y(G)}$, then $V(J_y(G))\not\subseteq V(G)$. Hence, there exists $(y,u)$ such that $y$ is a singular point of $g_u$ and $g_u(y)\ne0$. This gives
\[
\mu_I(f)\ge \beta_n(X_u)=\sum_{y\in B_\epsilon\setminus X_u}\mu(g_u;y)\ge 1.
\]
\end{proof}
%Me parece que en el Lema 3.7, es mejor no invocar los
%ideales radicales. El argumento no los requiere. En lugar de usarlos,
%pueden decir simplemente que si (y,v) es un punto crítico relativo de
%G con respecto a y,v entonces es un punto crítico relativo con respecto 
%a y   — eso es todo el significado de la inclusión 
%
%$$\sqrt{J_y(G)}\subseteq \sqrt{J_{y,v}(G)}$$
%
%que citais vosotros. Y al hablar más directamente de los puntos críticos 
%relativos, se ve que el argumento en efecto prueba que 
%
%$$\mu_I(h)\geq \mu_I(f)$$
%
%pues $\mu_I(f)$ ese número de puntos críticos relativos a $(y,v)$, y $\mu_I(h)$ 
%es el número de puntos críticos relativos a $y$, y hemos dicho que todos los
 %puntos críticos relativos a $(y,v)$ son  puntos críticos relativos a $y$.
%De acuerdo que es el mismo argumento, pero muestra algo
%más — no solo que si un \mu es diferente de cero entonces
%el otro también, sino que el segundo es mayor o igual que el
%primero. 
%
%Pregunto: estrictamente mayor, si el desdoblamiento es no trivial?
	
\begin{lemma}\label{segundolema}
Let $h\colon\CCS{n}\rightarrow \CCzero{n+1}$ be $\eqA$-finite and let $f$ be any unfolding of $h$ which is also $\eqA$-finite. If $\mu_I(f)>0$ then $\mu_I(h)>0$.
\end{lemma}

\begin{proof}
Assume that $f(x,v)=(h_v(x),v)$ and denote by $(y,v)$ the coordinates of $f$ in the target. Let $F$ be a stable unfolding of $f$. If $\mu_I(h)=0$, then $G\in\sqrt{J_y(G)}\subseteq \sqrt{J_{y,v}(G)}$, so $\mu_I(f)=0$.
\end{proof}	

%%As we were saying, a reformulation of \ref{milnor icis} based on \ref{mu k} and \ref{mu} will be indispensable in proving the following theorem. This is a special case of the inequality of Mond's conjecture [referencia?], that the $\eqA$-codimension is lower or equal than the image Milnor number. Also, that a map with zero image Milnor number was strongly believed to be stable but, again, there was no prove in the literature.

Let $f\colon\CCS{n}\rightarrow \CCzero{n+1}$ be $\eqA$-finite and assume that either $(n,n+1)$ are nice dimensions or $f$ has corank 1. 
%The Mond's Conjecture (cf. \cite{Mond1991}) says that
%\[
%\eqA_e\text{-}\codim(f)\le \mu_I(f),
%\]
%with equality if $f$ is weighted homogeneous. This conjecture is known to be true for $n=1,2$ (see \cite{deJong1991,Mond1991} for $n=2$ and \cite{Mond1995} for $n=1$) but it remains open for $n\ge 3$. 
Here we prove the following weak version of the Mond's Conjecture in the corank 1 case (see the introduction for the original version of the conjecture).

\begin{theorem}[Weak Mond's Conjecture]\label{wmc}
Let $f\colon\CCS{n}\rightarrow \CCzero{n+1}$ be $\eqA$-finite of corank 1. Then $\mu_I(f)=0$ if and only if $f$ is stable.
\end{theorem}
\begin{proof}
Obviously, $\mu_I(f)=0$ when $f$ is stable. Assume that $f$ is not stable. 
If $s(f)>d(f)$ we know that $d(f)=n+1$, and also that $\mu_{n+2}^\Alt(f)>0$. Hence, we can suppose that $s(f)\leq d(f)$.

By the Marar-Mond criterion either $D^k(f)$ is singular for some $k=2,\dots,d(f)$ or $D^k(f)$ is a $k$-tuple of points of $S$ for some $k\ge n+2$. We suppose first that $D^k(f)$ is singular, for some $k<n+1$. 

If $f$ admits a 1-parameter stable unfolding $F(x,t)=(f_t(x),t)$, then $H_{n-k+1}\left(D^k(f_t)\right)$ has non-zero alternating part for $t\ne0$, by Lemma \ref{milnor icis}. Since $D^k(F)$ is contractible and $k<n+1$, it follows from the exact sequence of the pair $\left(D^k(F),D^k(f_t)\right)$ that
\[
H^\Alt_{n-k+2}\left(D^k(F),D^k(f_t);\QQ\right) \cong H^\Alt_{n-k+1}\left(D^k(f_t);\QQ\right),
\]
so $\mu_k^\Alt(f)>0$.

If $f$ does not admit a 1-parameter stable unfolding, then we consider a minimal stable unfolding $F$. By taking a generic section on the parameter space, we get a finitely determined germ $F_0$ which is an unfolding of $f$ and which admits the 1-parameter stable unfolding $F$. Now $\mu_I(F_0)>0$ by the above argument and hence also $\mu_I(f)>0$ by Lemma \ref{segundolema}.

The next case to consider is when $D^{n+1}(f)$ is singular. Again, we use the exact sequence of the pair $\left(D^k(F),D^k(f_t)\right)$, but in this case 
\[
H_{1}^\Alt \left(D^{n+1}(F),D^{n+1}(f_t);\QQ\right)
\]
is isomorphic to the kernel of the mapping 
\begin{equation}\label{inclusion}
H_{0}^\Alt \left(D^{n+1}(f_t);\QQ\right)\longrightarrow H_{0}^\Alt \left(D^{n+1}(F);\QQ\right)
\end{equation}
induced by the inclusion. Take a singular 0-dimensional component of $D^{n+1}(f)$, with multiplicity $m>1$. Such component will split into $m$ distinct points in $D^{n+1}(f_t)$, which correspond to $m$ distinct generators of $H_{0}^\Alt \left(D^{n+1}(f_t);\QQ\right)$. But these $m$ points are in the same connected component of $D^{k+1}(F)$, for $F(x,t)=(f_t(x),t)$. Hence, we get a non-trivial element of the kernel of \eqref{inclusion} and thus $\mu_{n+1}^\Alt(f)>0$.

Finally, it only remains to consider the case where $D^{n+1}(f)$ is smooth but $D^k(f)$ is a $k$-tuple of points of $S$ for some $k\ge n+2$. Since $s(f)\le d(f)$, $D^{n+1}(f)$ necessarily must contain a point $\left(x^{(1)},\dots,x^{(n+1)}\right)$ such that $x^{(i)}=x^{(j)}$ for some $i\ne j$, as the projections from the previous $D^k(f)$ to this $D^{n+1}(f)$ cover all the possible points in the last space and we have less than $n+2$ points in $S$. This point will also split into several distinct points  in $D^{n+1}(f_t)$, which is not possible if $D^{k+1}(f)$ is smooth. We deduce that this case cannot occur when  $s(f)\le d(f)$.

\end{proof}

%\textcolor{red}{No entiendo bien el último caso}

\begin{note}
%The ideas used to prove this conjecture have non-trivial intersection with the ideas of Proposition 4.4 from \cite{Cooper2002}. In that proposition it is proved that a corank $1$ mono-germ of $\eqA_e$-codimension $1$ has image Milnor number equal to $1$, using the same result of Wall (Theorem \ref{wall iso}) and studying these kind of map-germs.
%\end{note}
The proof of Theorem \ref{wmc} is inspired in the proof of \cite[Proposition 4.4]{Cooper2002}. Here it is proved that a corank $1$ mono-germ of $\eqA_e$-codimension $1$ has image Milnor number equal to $1$, based on the same result of Wall (Theorem \ref{wall iso}).
\end{note}

The following corollary can be deduced easily from 
 Lemma \ref{milnor icis}, Theorem \ref{wmc} and their proofs and it gives a sharper estimate of $\mu_I(f)$ when $f$ is unstable.

\begin{corollary}
Let $f\colon\CCS{n}\rightarrow \CCzero{n+1}$ be $\eqA$-finite of corank 1 and unstable.
Assume $H_{n-k+1}(D^k(f_t);\QQ)$ has non-zero alternating part for some $k$:
\begin{enumerate}[label={(\roman*)}]
\item \label{it1} If $s(f)\leq d(f)$ then $\mu_I(f)\geq d(f)-k+1$.
\item \label{it2} If $s(f)>d(f)$  then $\mu_I(f)\geq d(f)-k+1+{s(f)-1 \choose d(f)}$.
\end{enumerate}
In case \ref{it1} there always exists such a $k$ and in case \ref{it2} $d(f)$ has to be equal to $n+1$ and such a $k$ could not exist.
\end{corollary}

A straightforward consequence of the weak Mond's conjecture is about the dimension of the relative Jacobian module of $f$ considered in \cite{Bobadilla2019}. It is defined as
\[
M_y(G)=\frac{J(G)+(G)}{J_y(G)},
\]
where $G\colon(\CC^{n+1}\times\CC^r,0)\to(\CC,0)$ is a function such that $G(y,u)=0$ is a reduced equation of the image of a stable unfolding of $f$. It is not difficult to see that the dimension of $M_y(G)$ is always $\le r$ when $f$ is $\eqA$-finite. Moreover, it is shown in \cite[Theorem 6.1]{Bobadilla2019} that the Mond's conjecture holds for $f$ when $M_y(G)$ is Cohen-Macaulay of dimension $r$.

\begin{corollary} Let $f\colon\CCS{n}\rightarrow \CCzero{n+1}$ be $\eqA$-finite of corank 1 and unstable. Then $M_y(G)$ has dimension $r$.
\end{corollary}

\begin{proof} It follows from \cite[Theorem 6.1]{Bobadilla2019} that 
\[
\mu_I(f)=e_{\mathcal O_r}\left((u_1,\dots,u_r);M_y(G)\right),
\] 
the Samuel multiplicity of the $\mathcal O_r$-module $M_y(G)$ with respect to the parameter ideal $(u_1,\dots,u_r)$. But it is well known that an $R$-module has multiplicity $>0$ if and only if it has dimension equal to $\dim R$.
\end{proof}

%\textcolor[rgb]{1,0,0}{
%En un futuro poner ejemplos}

%\textcolor{red}{There is also an easy result that can be deduced from the beginning of the proof of \ref{wmc}:
%\begin{proposition}
%In the conditions of \ref{wmc}, if $f$ has $d(f)<n+1$ and an unfolding with a minimal number of parameters $F$, then every generic section on the space of parameters of $F$ has the same image Milnor number.
%\end{proposition}
%Esto no me lo acabo de creer mucho, porque esto implicaría que todas las $f$ con el mismo desdoblamiento estable tienen el mismo $\mu_I(f)$.}

\section{Houston's conjecture on excellent unfoldings}\label{sect:hc}

It is not difficult to see that if we add a new branch to an unstable multi-germ $f\colon(\CC^n,S)\rightarrow(\CC^{n+1},0)$ then its  $\eqA_e$-codimension increases strictly (see for instance \cite[Exercise 3.4.1]{Mond-Nuno2020}). We show the same property for the image Milnor number, instead of the $\eqA_e$-codimension. The idea of the proof is easy to visualize, as we can see in fig.  \ref{fig:AddingBranches}.

\begin{figure}[ht]
	\centering
		\includegraphics[width=0.625\textwidth]{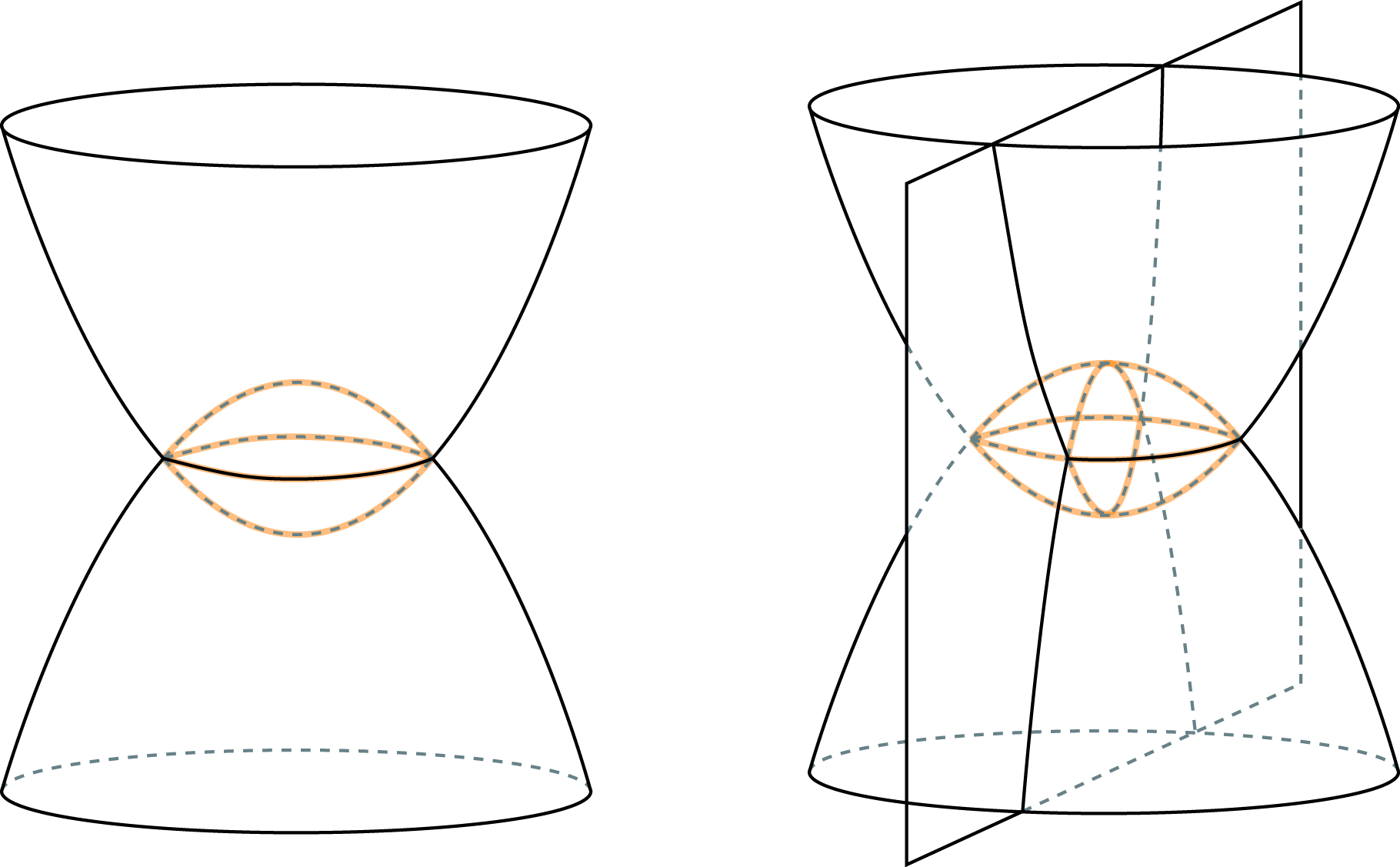}
	\caption{Real representation of the creation of more homology via the addition of more branches. Note that in the complex case this happens in middle dimension.}
	\label{fig:AddingBranches}
\end{figure}

Given two germs $f\colon(\CC^n,S)\rightarrow(\CC^{n+1},0)$ and $g\colon(\CC^n,z)\rightarrow(\CC^{n+1},0)$, we denote by $\{f,g\}\colon\left(\CC^n,S\sqcup \{z\}\right)\rightarrow(\CC^{n+1},0)$ the new multi-germ obtained as the disjoint union of $f$ and $g$. If $f$ and $g$ are both of corank 1 and $\eqA$-finite, then
\[
\mu_k(f)\le \mu_k\left(\{f,g\}\right),
\]
for all $k$, since adding a new branch does not kill the corresponding alternating homology of the $k$-multiple point space because the new branch just adds more connected components disjoint from the ones we had before. By Proposition \ref{mu}, this implies that
\[
\mu_I(f)\le \mu_I(\{f,g\}).
\]
We may have $\mu_I(f)=\mu_I(\{f,g\})$ when $f$ is stable and $g$ is transverse to $f$, so that $\{f,g\}$ is also stable. In the next lemma, we show that if $f$ is unstable, then the inequality is strict.

%To prove one of the fundamental results to solve the Houston's conjecture on excellent unfoldings we need a previous nice lemma about what happens if we add branches to a germ, the idea is obvious and it is represented in \ref{fig:AddingBranches} but it is non-trivial as attacking the middle homology studying directly the geometry is not easy.

\begin{lemma}\label{rama extra}
Let $f\colon(\CC^n,S)\rightarrow(\CC^{n+1},0)$ and $g\colon(\CC^n,z)\rightarrow(\CC^{n+1},0)$ be $\eqA$-finite. If $f$ has corank $1$ and $\mu_I(f)>0$ then 
\[
	\mu_I(f)<\mu_I(\left\{f,g\right\}).
\]
\end{lemma}
\begin{proof}
By the upper semi-continuity of the image Milnor number (see Corollary  \ref{upper}), we can assume that the image of $g$ is a generic hyperplane $H$ in $\CC^{n+1}$ through the origin. Let $f_t$ be a stable perturbation of $f$ with image $X_t$. Since $H$ is a generic hyperplane, the disjoint union $\{f_t,g\}$ gives a stable perturbation of $\{f,g\}$, with image $X_t\cup H$. 

Furthermore, $X_t\cap H$ is also the image of a stable perturbation of the restriction $\tilde f:\left(f^{-1}(H),S\right)\to(H,0)$. Since $H$ is generic and $f$ is $\eqA$-finite of corank 1, $\left(f^{-1}(H),S\right)$ is smooth and $\tilde f$ is also $\eqA$-finite of corank 1. Moreover, $\tilde f$ cannot be stable because $f$ is a 1-parameter unfolding of $\tilde f$. Hence $\mu_I(\tilde f)>0$, by the weak Mond's conjecture (Theorem \ref{wmc}).

Now, just apply the Mayer-Vietoris sequence:
\[
\xymatrix{0\ar[r] &H_n(X_t)\ar[r] &H_n(X_t\cup H)\ar[r] &H_{n-1}(X_t\cap H)\ar[r] &0},
\]
so
\[
\mu_I(\{f,g\})=\mu_I(f)+\mu_I(\tilde f)>\mu_I(f).
\]

\end{proof}

%\begin{remark}
%The proof of \ref{rama extra} gives an algorithm to find a lower bound of the increment of the Milnor number when a new branch is added. \textcolor[rgb]{1,0,0}{In fact these bounds are the best possible for a general $n$.}
%\end{remark}

We recall now the notion of excellent unfolding following Gaffney (cf. \cite[Definition 6.2]{Gaffney1993}). Excellent unfoldings play an important role in the theory of equisingularity of families of germs. In fact, when $F$ is excellent then we can stratify $F$ in such a way that the parameter axes in the source and target are the only 1-dimensional strata (see fig. \ref{fig:Excellent unfolding}).

\begin{definition}\label{def:excellent}
A one-parameter origin-preserving unfolding $F$ is called excellent if it is good and it has a representative as in Definition \ref{good} such that, in addition, $f_t$ has no 0-stable singularities on $W\setminus\{0\}$ (i.e., stable singularities whose isosingular locus is 0-dimensional).
\end{definition}

%The main idea of this definition is explained in fig. \ref{fig:Excellent unfolding}.

\begin{figure}[ht]
	\centering
		\includegraphics[width=1.00\textwidth]{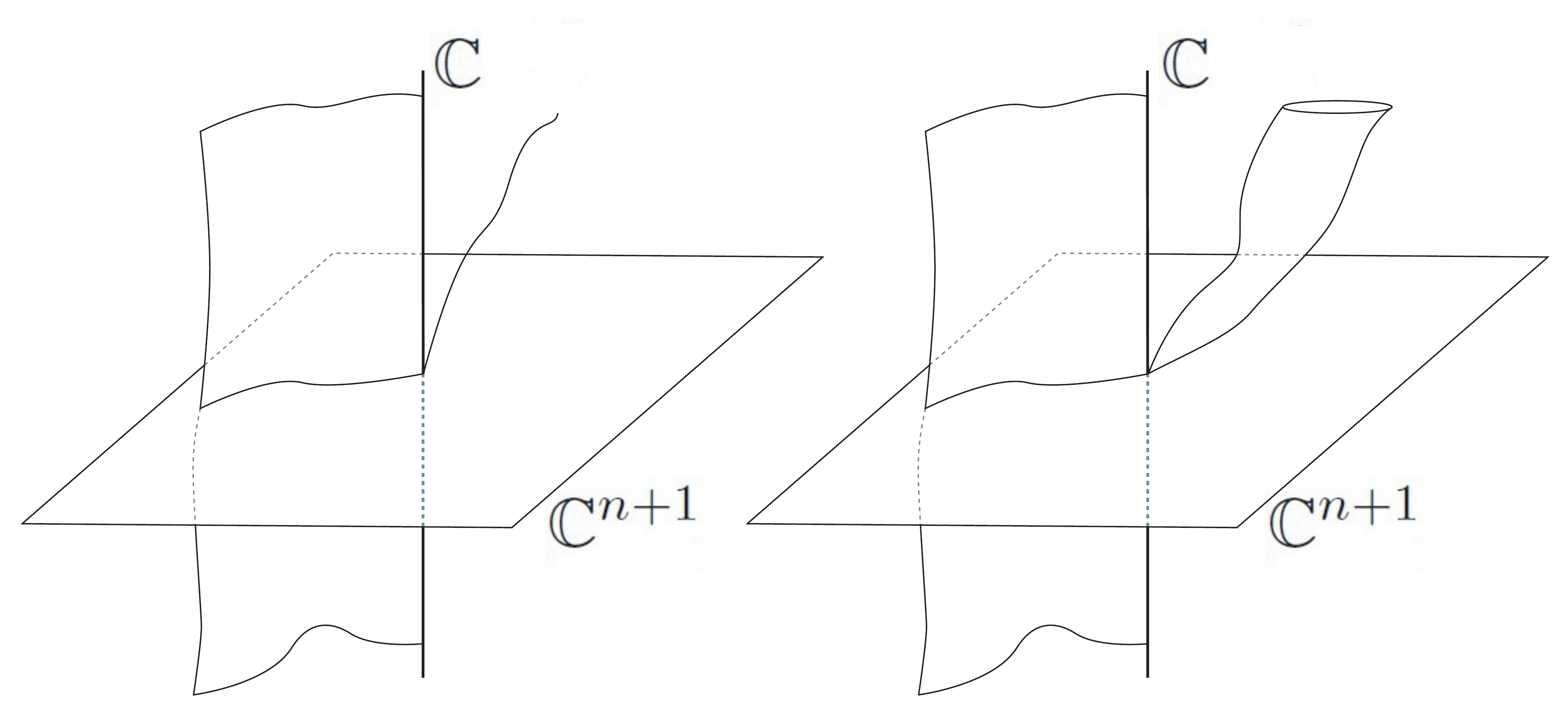}
	\caption{The pictures show the stratifications of the image of a non-excellent unfolding (left),  due to the presence of a 1-dimensional stratum distinct from the parameter axis (bold line), and an excellent unfolding (right).}
%	
%	Here we have depictured the images of $f_t$ along the $\CC$ axis, the parameter space. If we have a stratum of dimension $1$ of non-stable points (left) then $f_t$ is not a good unfolding. If this strata is of stable points and so is every other strata then the unfolding is good but not excellent. An excellent unfolding could be the picture on the right, where all the strata have dimension greater than 1 and are made of stable-type points.}
	\label{fig:Excellent unfolding}
\end{figure}

The above lemma together with the conservation of the image Milnor number and the weak Mond's conjecture allow us to prove Houston's Conjecture on excellent unfoldings (cf. \cite[Conjecture 6.2]{Houston2010}) which we state now.

\begin{theorem}\label{hc}
Let $f:\GS{n}{n+1}$ be $\eqA$-finite of corank $1$ and  let $F(x,t)=(f_t(x),t)$ be an origin-preserving one-parameter unfolding. Consider the family of germs $f_t\colon\GS{n}{n+1}$.
Then $\mu_I(f_t)$ constant implies $F$ excellent.
\end{theorem}
\begin{proof}
We will use \cite[Corollary 5.9]{Houston2010}, so we only need to show that $F$ is good and that either 
 $s(\hat f_t)\leq d(\hat f_t)$ for all $t$ or $s(\hat f_t)$ and $d(\hat f_t)$ are both constant, where $\hat f_t$ is the germ at $f_t^{-1}(0)$ (we keep the notation $f_t$ for the germ at $S$).
 
We can suppose that $f$ is not stable, otherwise the result is trivial.
We first prove that $s(\hat f_t)$ is constant, that is, $f_t^{-1}(0)=S$ and hence, $\hat f_t=f_t$. We have $S\subseteq f_t^{-1}(0)$ and if the inclusion was strict, then $\mu_I(f_t)<\mu_I(\hat f_t)$ by Lemma \ref{rama extra}. But the upper semi-continuity of Corollary \ref{upper} implies that $\mu(\hat f_t)\le \mu_I(f)$, in contradiction with the constancy of $\mu_I(f_t)$.

If $s(f_{t_0})>d(f_{t_0})$ for some $t_0$ then this can only happen when $d(f_{t_0})=n+1$. But $s(f_t)$ is constant so $s(f_t)>n+1\geq d(f_t)$, and again we have $d(f_t)=n+1$. This shows that either $s(f_t)\leq d(f_t)$ for all $t$ or $s(f_t)$ and $d(f_t)$ are both constant.

Finally, we use the conservation of the image Milnor number, Theorem \ref{conservation}, to show that $F$ is good. In fact, we get
\[
\mu_I(f_t;0)=\mu_I(f)\ge \sum_{y\in X_t}\mu_I(f_t;y),
\]
so $\mu_I(f_t;y)=0$ for all $y\in X_t\setminus\{0\}$. By the weak Mond's conjecture Theorem \ref{wmc}, $f_t$ is locally stable on $X_t\setminus\{0\}$.
\end{proof}

One can ask if the converse is true, that is, if an excellent unfolding implies constant image Milnor number. We have the following partial result:

\begin{proposition}\label{converse}
Let $f:\GS{n}{n+1}$ be $\eqA$-finite with $n=1,2$ and  let $F(x,t)=(f_t(x),t)$ be an origin-preserving one-parameter unfolding.
Then $F$ excellent implies $\mu_I(f_t)$ constant.
\end{proposition}

\begin{proof} Let $n=1$. We have $\mu_I(f_t)=\delta(f_t)-s(f_t)+1$, where $\delta(f_t)$ is the delta invariant (see, for example, \cite[Lemma 2.2]{Mond1995}). Obviously $s(f_t)=|S|$ is constant and we also have conservation of the delta invariant, which means that
\[
\delta(f)=\sum_{y\in \Sigma(X_t)} \delta(f_t;y),
\]
where $\Sigma(X_t)$ is the singular locus of the image of $f_t$ and  $\delta(f_t;y)$ is the delta invariant of the germ of $f_t$ at $f^{-1}_t(y)$. Since $F$ is excellent, we have $\Sigma(X_t)=\{0\}$ and $f^{-1}_t(0)=S$, so $\delta(f_t)=\delta(f_t;0)$ is also constant.

Let $n=2$. We consider the double point curve in the source $D(f_t)$, defined as $p_1(D^2(f_t))$, where $p_1:\CC^2\times\CC^2\rightarrow\CC^2$ is the projection onto the first component. Then $D(f_t)$ is a family of germs of plane curves in $(\CC^2,S)$. Since $F$ is excellent, we can choose representatives of $D(f_t)$ on some open neighbourhood $U$ of $S$ in $\CC^2$ such that $\Sigma(D(f_t))$ is equal to $S$ for all $t$. This implies that the (usual) Milnor number $\mu(D(f_t);x)$ at each point $x\in S$ must be constant. By a theorem of Fernández de Bobadilla and Pe-Pereira, cf. \cite[Theorem C]{FernandezdeBobadilla2008}, the unfolding $F$ is topologically trivial. So, $\mu_I(f_t)$ is constant by Corollary \ref{top-triv-2}.
\end{proof}

\begin{example}
The family $f_t (x,y) = \left(x,y^2,yp_t (x,y)\right)$ with 
$$ p_t(x,y)=\left(x-\frac{t}{2}\right)^2+\left(y^2-\frac{t}{2}\right)^2-\frac{t^2}{8} $$
yields an excellent unfolding over $\RR$, but not over $\CC$ because $y=0$ and $x=\frac{1}{2}\left(t\pm\frac{i}{2}\sqrt{t^2}\right) $ are curves of non-immersive points of $f_t$. Furthermore its image Milnor number is not constant, $\mu_I(f_0)>\mu_I(f_t)$ for $t\neq 0$ (cf. fig. \ref{fig:corank2counterexample}).

\begin{figure}[htbp]
	\centering
		\includegraphics[width=1.00\textwidth]{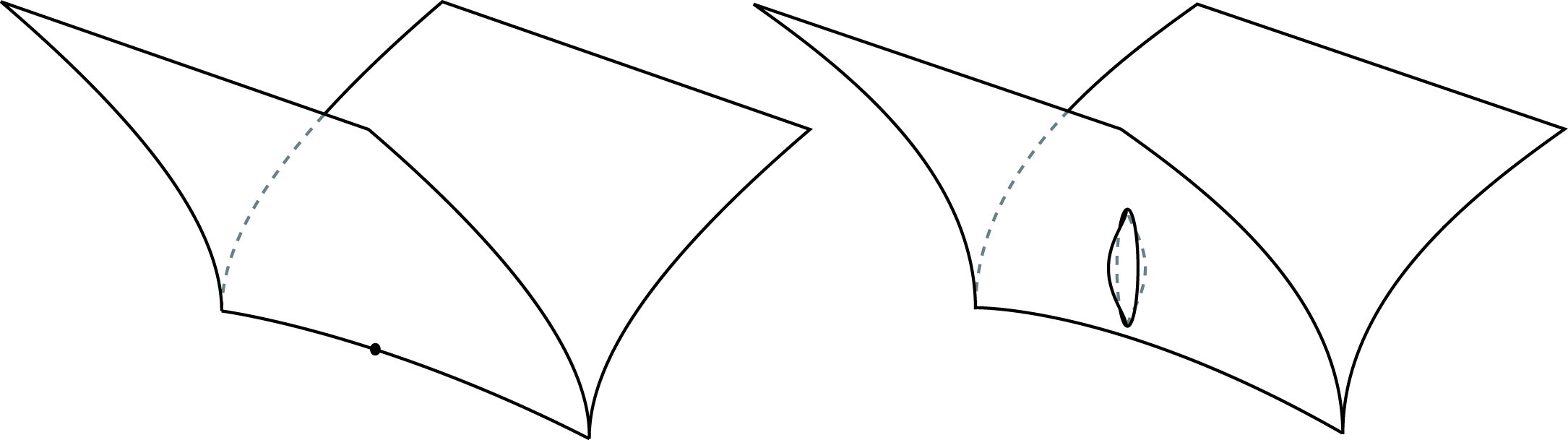}
	\caption{From left to right, $f_0$ and $f_t$ with $t\neq 0$ as real maps.}
	\label{fig:corank2counterexample}
\end{figure}

\end{example}

%Having this in sight we conjecture that there is an holomorphic-geometric restriction that prevents to find a counterexample.

Theorem \ref{hc} and Proposition \ref{converse} motivate the following more general conjecture, where we consider not only the converse of \ref{hc} in higher dimensions, but also drop the corank 1 condition.

\begin{conjecture}
For every $f:\GS{n}{n+1}$ $\eqA$-finite germ, and every $F(x,t)=(f_t(x),t)$ origin-preserving one-parameter unfolding, $F$ is excellent if and only if $\mu_I(f_t)$ is constant.
\end{conjecture}

\bibliographystyle{plain}
\bibliography{Mybib}

\end{document}